\documentclass[a4paper]{amsart}
\numberwithin{equation}{section}
\newtheorem{theorem}{Theorem}[section]
\newtheorem{corollary}{Corollary}[section]
\newtheorem{definition}{Definition}[section]
\newtheorem{lemma}{Lemma}[section]

\theoremstyle{remark}

\usepackage{hyperref}
\usepackage{graphicx}
\title[On Booth lemniscate of starlike functions]
 {On Booth lemniscate of starlike functions}
\begin{document}
\begin{abstract}
Assume that $\Delta$ is the open unit disk in the complex plane and
$\mathcal{A}$ is the class of normalized analytic functions in
$\Delta$. In this paper we introduce and study the class
\begin{equation*}
  \mathcal{BS}(\alpha):=\left\{f\in \mathcal{A}: \left(\frac{zf'(z)}{f(z)}-1\right)\prec \frac{z}{1-\alpha z^2}, \, z\in\Delta\right\},
\end{equation*}
where $0\leq\alpha\leq1$ and $\prec$ is the subordination relation.
Some properties of this class like differential subordination,
coefficients estimates and Fekete-Szeg\"{o} inequality associated
with the $k$-th root transform are considered.
\end{abstract}
\author[R. Kargar, A. Ebadian and J. Sok\'{o}{\l}]
       {Rahim Kargar, Ali Ebadian and Janusz Sok\'{o}{\l} }
\address{Department of Mathematics, Payame Noor University, Tehran, Iran}
       \email {rkargar@pnu.ac.ir {\rm (Rahim Kargar)}}
\address{Department of Mathematics, Payame Noor University, Tehran, Iran}
       \email {ebadian.ali@gmail.com {\rm (Ali Ebadian)}}
\address{ University of Rzesz\'{o}w, Faculty of Mathematics and Natural
         Sciences, ul. Prof. Pigonia 1, 35-310 Rzesz\'{o}w, Poland}
       \email{jsokol@ur.edu.pl {\rm (Janusz Sok\'{o}{\l})}}
\maketitle
\section{Introduction}

Let $\mathcal{A}$ denote the class of functions $f(z)$ of the
form:
\begin{equation}\label{f}
    f(z)=z+ \sum_{n=2}^{\infty}a_{n}z^{n},
\end{equation}
which are analytic and normalized in the open unit disk $\Delta=\{z\in \mathbb{C} :
|z|<1\}$. The subclass of
$\mathcal{A}$ consisting of all univalent functions $f(z)$ in
$\Delta$ is denoted by $\mathcal{S}$. A function $f\in\mathcal{S}$
is called starlike (with respect to $0$), denoted by
$f\in\mathcal{S}^*$, if $tw\in f(\Delta)$ whenever $w\in f(\Delta)$
and $t\in[0, 1]$. Robertson introduced in \cite{ROB}, the class $\mathcal{S}^*(\gamma)$ of starlike functions of order $\gamma\leq1$,
which is defined by
\begin{equation*}
   \mathcal{S}^*(\gamma):=\left\{ f\in \mathcal{A}:\ \ \mathfrak{Re} \left\{\frac{zf'(z)}{f(z)}\right\}> \gamma, \ z\in
    \Delta\right\}.
\end{equation*}
If $\gamma\in[0,1)$, then a function in $\mathcal{S}^*(\gamma)$ is univalent. In particular we put
$\mathcal{S}^*(0)\equiv \mathcal{S}^*$. We denote by $\mathfrak{B}$ the class of analytic functions $w(z)$ in
$\Delta$ with $w(0) = 0$ and $|w(z)| < 1$, $(z \in \Delta)$.
If $f$ and $g$ are two of the functions in $\mathcal{A}$, we say
that $f$ is subordinate to $g$, written $f (z)\prec g(z)$, if there
exists a $w\in\mathfrak{B}$ such that $f (z) = g(w(z))$, for all
$z\in\Delta$.

Furthermore, if the function $g$ is univalent in $\Delta$, then
we have the following equivalence:
\begin{equation*}
    f (z)\prec g(z) \Leftrightarrow (f (0) = g(0)\quad {\rm and}\quad f (\Delta)\subset g(\Delta)).
\end{equation*}
We now recall from \cite{psok}, a one-parameter family of functions as follows:
\begin{equation}\label{Falpha}
  F_{\alpha}(z):=\frac{z}{1-\alpha z^2}=z+\sum_{n=1}^{\infty}\alpha^{n}z^{2n+1}\qquad (z\in\Delta,~0\leq \alpha\leq1).
\end{equation}
The function $ F_{\alpha}(z)$ is starlike univalent for $0\leq \alpha<1$.
We have also $F_{\alpha}(\Delta)=D(\alpha)$, where
\begin{equation*}\label{D(alpha)}
  D(\alpha)=\left\{x+iy\in\mathbb{C}: ~ \left(x^2+y^2\right)^2-\frac{x^2}{(1-\alpha)^2}-\frac{y^2}{(1+\alpha)^2}<0\right\},
\end{equation*}
when $0\leq \alpha<1$ and
\begin{equation*}\label{D(1)}
  D(1)=\left\{x+iy\in\mathbb{C}: ~ \left(\forall t\in (-\infty,-i/2]\cup [i/2,\infty)\right)[x+iy\neq it]\right\}.
\end{equation*}

\begin{figure}[htp]
 \centering
 \includegraphics[width=7cm]{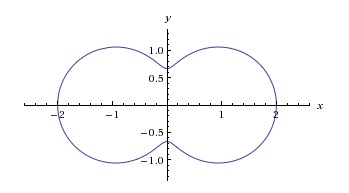}\\
 \caption{The boundary curve of $D(1/2)$}\label{Fig:1}
\end{figure}

The Persian curve (cf. \cite{AAS}) is a plane algebraic curve of order four that is the line
of intersection between the surface of a torus and a plane parallel to its axis. The equation in rectangular coordinates is
\begin{equation*}\label{Persian curve}
  \left(x^2+y^2+p^2+d^2-r^2\right)^2=4d^2\left(x^2+p^2\right),
\end{equation*}
where $r$ is the radius of the circle describing the torus, $d$ is
the distance from the origin to its center and $p$ is the distance
from the axis of the torus to the plane. We remark that a curve
described by
\begin{equation*}\label{Blemn}
    \left(x^2 + y^2\right)^2 - \left(n^4 + 2m^2\right)x^2 - \left(n^4 - 2m^2\right)y^2 = 0\qquad (x, y)\neq(0, 0),
\end{equation*}
is a special case of Persian curve that studied by Booth and is called the Booth lemniscate \cite{Booth}.
The Booth lemniscate is called elliptic if $n^4 > 2m^2$ while, for $n^4 < 2m^2$, it is termed hyperbolic. Thus
it is clear that the curve
\begin{equation*}
  \left(x^2+y^2\right)^2-\frac{x^2}{(1-\alpha)^2}-\frac{y^2}{(1+\alpha)^2}=0\qquad (x, y)\neq(0, 0),
\end{equation*}
is the Booth lemniscate of elliptic type (see figure 1). Two other special case of Persian curve are Cassini oval and Bernoulli lemniscate.

A plane algebraic curve of order four whose equation in Cartesian coordinates has the form:
\begin{equation*}
  \left(x^2 + y^2\right)^2 - 2c^2 \left(x^2 -y^2\right) = a^4 -c^4.
\end{equation*}

The Cassini oval is the set of points such that the product of the
distances from each point to two given points $p_2=(-c,0)$ and
$p_1=(c,0)$ (the foci) is constant. When $a \geq c\sqrt{2}$ the
Cassini oval is a convex curve; when $c<a<c\sqrt{2}$ it is a curve
with "waists" (concave parts); when $a=c$ it is a Bernoulli
lemniscate; and when $a<c$ it consists of two components. Cassini
ovals are related to lemniscates. Cassini ovals were studied by G.
Cassini (17th century) in his attempts to determine the Earth's
orbit.

The Bernoulli lemniscate plane algebraic curve of order four, the equation of which in orthogonal Cartesian coordinates is
\begin{equation*}
  \left(x^2 + y^2\right)^2 - 2a^2 \left(x^2 -y^2\right) =0,
\end{equation*}
and in polar coordinates
\begin{equation*}
  \rho^2=2a^2 \cos 2 \phi.
\end{equation*}

The Bernoulli lemniscate is symmetric about the coordinate origin, which is a node with tangents
$y=\pm x$ and the point of inflection. The product of the distances of any point $M$ to the two given points
$p_1=(-a,0)$ and $p_2=(a,0)$ is equal to the square of the distance between the points $p_1$ and $p_2$.
The Bernoulli lemniscate is a special case of the Cassini ovals, the lemniscates, and the sinusoidal spirals.
The Bernoulli spiral was named after Jakob Bernoulli, who gave its equation in 1694.

In \cite{KES(Complex)}, the authors introduced and studied the class $\mathcal{M}(\delta)$ as follows:\\
{\bf Definition A.} Let $\pi/2\leq \delta<\pi$. Then the function
$f\in\mathcal{A}$ belongs to the class $\mathcal{M}(\delta)$ if $f$
satisfies:
\begin{equation}\label{1definition}
    1+\frac{\delta-\pi}{2 \sin \delta}<
    \mathfrak{Re}\left\{\frac{zf'(z)}{f(z)}\right\} <
    1+\frac{\delta}{2\sin \delta} \qquad (z\in\Delta).
\end{equation}
By definition of subordination  and by \eqref{1definition}, we have that $f\in
\mathcal{M}(\delta)$ if and only if
\begin{equation*}
\left(\frac{z f'(z)}{f(z)}-1\right)\prec\mathcal{B}_\alpha(z):=
\frac{1}{2i\sin\delta}\log\left(\frac{1-z}{1-ze^{-i\delta}}\right)
\qquad (z\in\Delta),
\end{equation*}
where $\pi/2\leq \delta<\pi$. The above function $\mathcal{B}_\alpha(z)$ is convex
univalent in $\Delta$ and maps
$\Delta$ onto $\Gamma_\delta=\{w: (\delta-\pi)/(2\sin \delta)<\mathfrak{Re}\{w\}<\delta/(2\sin \delta)\}$,
or onto the convex hull of three points (one of which may be that point
at infinity) on the boundary of $\Gamma_\delta$. In other words, the image of $\Delta$ may be a vertical
strip when $\pi/2\leq \delta<\pi$, while in other cases, a half strip, a trapezium, or a triangle.

It was proved in \cite{KuOwa}, that for $\alpha<1<\beta$, the following
function $P_{\alpha,\beta}:\Delta\rightarrow \mathbb{C}$, defined by
\begin{equation}\label{P}
   P_{\alpha,\beta}(z)=1+\frac{\beta-\alpha}{\pi}i \log \left(\frac{1-e^{2\pi
   i\frac{1-\alpha}{\beta-\alpha}}z}{1-z}\right)\qquad (z\in\Delta),
\end{equation}
maps $\Delta$ onto a convex domain
\begin{equation}\label{Omega}
    P_{\alpha,\beta}(\Delta)=\{ w\in \mathbb{C}: \alpha<
    \mathfrak{Re}\{w\}<\beta\},
\end{equation}
conformally. Therefore, the function $P_{\alpha,\beta}(z)$ defined
by \eqref{P} is convex univalent in $\Delta$ and has the form:
\begin{equation*}
  P_{\alpha,\beta}(z)=1+\sum_{n=1}^{\infty} B_n z^n,
\end{equation*}
where
\begin{equation}\label{B-n}
B_n=\frac{\beta-\alpha}{n\pi}i \left(1-e^{2n\pi
i\frac{1-\alpha}{\beta-\alpha}}\right)\qquad (n=1,2,\ldots).
\end{equation}

The present authors (see \cite{KES(Siberian)}) introduced the class $\mathcal{V}(\alpha,\beta)$ as follows:\\
{\bf Definition B.} Let $\alpha<1$ and $\beta>1$. Then the
function $f\in\mathcal{A}$ belongs to the class
$\mathcal{V}(\alpha,\beta)$ if $f$ satisfies:
\begin{equation}\label{Bdefinition}
    \alpha<
    \mathfrak{Re}\left\{\left(\frac{z}{f(z)}\right)^2 f'(z)\right\} <\beta \qquad (z\in\Delta).
\end{equation}
Therefore, by definition of subordination, we have
that $f\in \mathcal{V}(\alpha,\beta)$ if and only if
\begin{equation*}
  \left(\frac{z}{f(z)}\right)^2 f'(z)
  \prec P_{\alpha,\beta}(z)\qquad (z\in \Delta).
\end{equation*}

Motivated by Definition A, Definition B and using $F_\alpha$, we
introduce a new class. Our principal definition is the following.
\begin{definition}\label{lemmaiff}
  Let $f\in \mathcal{A}$ and $0\leq \alpha< 1$.
Then $f\in \mathcal{BS}(\alpha)$ if and only if
  \begin{equation}\label{defsoblem}
      \left(\frac{zf'(z)}{f(z)}-1\right)\prec F_\alpha(z)\qquad (z\in\Delta),
  \end{equation}
  where $F_\alpha$ defined by \eqref{Falpha}.
\end{definition}
In our investigation, we require the following result.

\begin{corollary}\label{c7}
We have that $f\in \mathcal{BS}(\alpha)$ if and only if
\begin{equation}\label{1c7}
    f(z)=z\exp\int_0^z \frac{F_{\alpha}(w(t))-1}{t}{\rm d}t\qquad (z\in\Delta),
\end{equation}
for some function $w(z)$, analytic in $\Delta$, with $|w(z)|\leq
|z|$ in $\Delta$.
\end{corollary}
\begin{proof}
From \eqref{defsoblem} it follows that there exists a function
$w(z)$, analytic in $\Delta$, with $|w(z)|\leq |z|$ in $\Delta$,
such that
\begin{equation*}
      z\left(\frac{f'(z)}{f(z)}-\frac{1}{z}\right)=F_\alpha(w(z))\qquad (z\in\Delta),
\end{equation*}
or
\begin{equation*}
      z\left(\log\frac{f(z)}{z}\right)'=F_\alpha(w(z))\qquad
      (z\in\Delta).
\end{equation*}
This gives \eqref{1c7}. On the other hand, it is a easy calculation
that a function having the form \eqref{1c7} satisfies condition
\eqref{defsoblem}.
\end{proof}

Applying formula \eqref{1c7} for $w(z)=z$ gives that
\begin{equation}\label{f_0}
    f_0(z)=z\left(\frac{1+\sqrt{\alpha}z}{1-\sqrt{\alpha}z}\right)^{1/\sqrt{\alpha}}
    \qquad (z\in\Delta),
\end{equation}
is in the class $ \mathcal{BS}(\alpha)$.

\begin{lemma}\label{lemReF}
Let $F_{\alpha}(z)$ be given by \eqref{Falpha}. Then
\begin{equation}\label{ReFalpha}
    \frac{1}{\alpha-1}< \mathfrak{Re}\left\{F_{\alpha}(z)\right\}< \frac{1}{1-\alpha}\qquad (0\leq \alpha<1).
\end{equation}
\end{lemma}
\begin{proof}
If $\alpha=0$, then we have
$-1<\mathfrak{Re}\{F_\alpha\}=\mathfrak{Re}(z)<1$. For $0<\alpha<1$,
the function $\{F_\alpha\}$ does not have any poles in
$\overline{\Delta}$ and is analytic in $\Delta$, thus looking for
the $\min\{\mathfrak{Re}\{F_\alpha(z)\}:~ |z|<1\}$ it is sufficient
to consider it on the boundary $\partial
F_\alpha(\Delta)=\{F_\alpha(e^{i\varphi}):\varphi \in [0,2\pi]\}$. A
simple calculation give us
\begin{equation*}\label{ReFpsok}
    \mathfrak{Re}\left\{F_{\alpha}(e^{i\varphi})\right\}=\frac{(1-\alpha)\cos\varphi}{1+\alpha^2-2\alpha\cos2\varphi}
    \qquad(\varphi\in[0,2\varphi]).
\end{equation*}
So we can see that $\mathfrak{Re}\left\{F_{\alpha}(z)\right\}$ is
well defined also for $\varphi=0$ and $\varphi=2\pi$. Define
\begin{equation*}\label{g(x)}
    g(x)=\frac{(1-\alpha)x}{1+\alpha^2-2\alpha (2x^2-1)}\qquad (-1\leq x\leq 1),
\end{equation*}
then for $0< \alpha<1$, we have $g'(x)>0$. Thus for $-1\leq x\leq
1$, we have
\begin{equation*}
    \frac{1}{\alpha-1}=g(-1)\leq g(x)\leq g(1)=\frac{1}{1-\alpha}.
\end{equation*}
This completes the proof.
\end{proof}
We note that from Lemma \ref{lemReF} and by definition of subordination, the function $f\in\mathcal{A}$ belongs to
the class $\mathcal{BS}(\alpha)$, $0\leq \alpha<1$, if it satisfies the condition
\begin{equation*}\label{defsob}
      \frac{1}{\alpha-1}< \mathfrak{Re}\left(\frac{zf'(z)}{f(z)}-1\right)< \frac{1}{1-\alpha}\qquad (z\in\Delta),
\end{equation*}
or equivalently
    \begin{equation}\label{defsob2}
      \frac{\alpha}{\alpha-1}< \mathfrak{Re}\left(\frac{zf'(z)}{f(z)}\right)<\frac{2-\alpha}{1-\alpha}\qquad (z\in\Delta).
  \end{equation}
It is clear that $\mathcal{BS}(0)\equiv \mathcal{S}(0,2)\subset
\mathcal{S}^*$, where the class $\mathcal{S}(\alpha, \beta)$,
$\alpha< 1$ and $\beta > 1$, was recently considered by K. Kuroki
and S. Owa in \cite{KuOwa}.

\begin{corollary}\label{c0}
If $f\in \mathcal{BS}(\alpha)$, then
\begin{equation}\label{1c0}
      \frac{zf'(z)}{f(z)}\prec P_\alpha(z) \qquad (z\in\Delta),
\end{equation}
where
\begin{equation}\label{2c0}
    P_{\alpha}(z)=1+\frac{2}{\pi(1-\alpha)}i \log
    \left(\frac{1-e^{\pi
    i(1-\alpha)^2}z}{1-z}\right)\qquad (z\in\Delta),
\end{equation}
is convex univalent in $\Delta$.
\end{corollary}
\begin{proof}
If $f\in \mathcal{BS}(\alpha)$, then it satisfies \eqref{defsob2} or
$zf'(z)/f(z)$ lies in a strip of the form \eqref{Omega}. Then
applying the definition of subordination and  the function
\eqref{P}, we obtain \eqref{1c0} and \eqref{2c0}.
\end{proof}
For the proof of our main results, we need the following Lemma.

\begin{lemma}\label{lem1.3}(See \cite{Rog})
Let $q(z)=\sum_{n=1}^{\infty}C_nz^n$ be analytic and
univalent in $\Delta$, and suppose that $q(z)$ maps $\Delta$
onto a convex domain. If $p(z) = \sum_{n=1}^{\infty}A_nz^n$ is
analytic in $\Delta$ and satisfies the following subordination
\begin{equation*}
p(z)\prec q(z)\qquad (z\in\Delta),
\end{equation*}
then
\begin{equation*}
|A_n|\leq |C_1|\qquad n\geq 1.
\end{equation*}
\end{lemma}


\section{ Main Results}
The first main result is the following theorem.
\begin{theorem}\label{Th.f(z)/z}
Let $f\in \mathcal{A}$ and $0\leq\alpha<1$. If $f\in
\mathcal{BS}(\alpha)$ then
\begin{equation}\label{11Th.f(z)/z}
    \log\frac{f(z)}{z}\prec\int_0^z \frac{P_{\alpha}(t)-1}{t}{\rm d}t\qquad (z\in\Delta),
\end{equation}
where
\begin{equation*}
    P_{\alpha}(z)-1=\frac{2}{\pi(1-\alpha)}i \log
    \left(\frac{1-e^{\pi
    i(1-\alpha)^2}z}{1-z}\right)\qquad (z\in\Delta)
\end{equation*}
and
\begin{equation*}
    \widetilde{P}_{\alpha}(z)=\int_0^z \frac{P_{\alpha}(t)-1}{t}{\rm d}t\qquad (z\in\Delta),
\end{equation*}
are convex univalent in $\Delta$.
\end{theorem}

\begin{proof}

If $f\in \mathcal{BS}(\alpha)$, then by \eqref{1c0} it satisfies
\begin{equation}\label{2Th.f(z)/z}
    z\left\{\log\frac{f(z)}{z}\right\}'\prec P_\alpha(z)-1.
\end{equation}
It is known that if $\mathcal{F}(z)$ is convex univalent in $\Delta$, then
\begin{equation}\label{3Th.f(z)/z}
     \left[f(z)\prec \mathcal{F}(z)\right]\Rightarrow\left[\int_0^z \frac{f(t)}{t}{\rm d}t\prec \int_0^z \frac{\mathcal{F}(t)}{t}{\rm d}t\right]
\end{equation}
and
\begin{equation*}
    \widetilde{F}(z)=\int_0^z \frac{\mathcal{F}(t)}{t}{\rm d}t,
\end{equation*}
is convex univalent in $\Delta$. By Corollary \ref{c0}, we know that
$P_\alpha(z)-1$ is  convex univalent in $\Delta$. Therefore,
applying \eqref{3Th.f(z)/z} in \eqref{2Th.f(z)/z} gives \eqref{11Th.f(z)/z} with convex univalent
$\widetilde{P}_{\alpha}(z)$.
\end{proof}

\begin{corollary}\label{c1}
If $f\in \mathcal{BS}(\alpha)$ and $|z|=r<1$, then
\begin{equation}\label{1c1}
      \min_{|z|=r}\left|\exp\widetilde{P}_{\alpha}(z)\right|\leq\left|\frac{f(z)}{z}\right|
      \leq\max_{|z|=r}\left|\exp\widetilde{P}_{\alpha}(z)\right|.
\end{equation}
\end{corollary}
\begin{proof}
Subordination \label{1Th.f(z)/z} implies
\begin{equation}\label{2c1}
    \frac{f(z)}{z}\prec\exp\widetilde{P}_{\alpha}(z)
\end{equation}
and $\exp\widetilde{P}_{\alpha}(z)$ is convex univalent. Then
\eqref{2c1} implies  \eqref{1c1}.
\end{proof}

We now obtain coefficients estimates for functions belonging to the class $\mathcal{BS}(\alpha)$.
\begin{theorem}\label{t2.1}
 Assume that the function $f$ of the form \eqref{f} belongs to the class $\mathcal{BS}(\alpha)$ where $0\leq \alpha\leq 3-2\sqrt{2}$. then $|a_2|\leq 1 $ and
 \begin{equation}\label{est an}
   |a_n|\leq \frac{1}{n-1}\prod_{k=2}^{n-1}\left(\frac{k}{k-1}\right)\qquad (n=3,4,\ldots).
 \end{equation}
\end{theorem}
\begin{proof}
  Assume that $f\in\mathcal{BS}(\alpha)$. Then from Definition \ref{lemmaiff} we have
  \begin{equation}\label{1proof}
    p(z)\prec 1+ F_\alpha(z)=1+z+\alpha z^3+\cdots\qquad (z\in\Delta),
  \end{equation}
  where
  \begin{equation}\label{2proof}
    zf'(z)=p(z)f(z).
  \end{equation}
  We note that $F_\alpha$ is convex function for $0\leq \alpha\leq 3-2\sqrt{2}$ (see \cite[Corollary 3.3]{psok}).
  If we define $p(z)=1+\sum_{n=1}^{\infty}p_nz^n$, then from Lemma \ref{lem1.3}, we have
  \begin{equation}\label{3proof}
    |p_n|\leq 1.
  \end{equation}
    Equating the coefficients of $z^n$ on both sides of \eqref{2proof}, we find the following relation between the coefficients:
  \begin{equation}\label{4proof}
    na_n=p_{n-1}+a_2p_{n-2}+\cdots+a_{n-1}p_1+a_n.
  \end{equation}
  Making use of \eqref{3proof} and \eqref{4proof}, we get
  \begin{equation}\label{5proof}
    |a_n|\leq\frac{1}{n-1}\sum_{k=1}^{n-1}|a_{k}|\qquad |a_1|=1.
  \end{equation}
  Obvious that, from \eqref{5proof} we have $|a_2|\leq 1$. We now need show that
  \begin{equation}\label{6proof}
    \frac{1}{n-1}\sum_{k=1}^{n-1}|a_{k}|\leq \frac{1}{n-1}\prod_{k=2}^{n-1}\left(1+\frac{1}{k-1}\right)\qquad (n=3,4,\ldots).
  \end{equation}
  We use induction to prove \eqref{6proof}. If we take $n=3$ in the inequality \eqref{6proof}, we have $|a_2|\leq 1$, therefore the case
  $n=3$ is clear. A simple calculation gives us
  \begin{align*}
    |a_{m+1}| &\leq  \frac{1}{m}\sum_{k=1}^{m}|a_{k}|=\frac{1}{m}\left(\sum_{k=1}^{m-1}|a_{k}|+|a_m|\right)\\
    &\leq \frac{1}{m}\prod_{k=2}^{m-1}\left(1+\frac{1}{k-1}\right)+\frac{1}{m}\times
    \frac{1}{m-1}\prod_{k=2}^{m-1}\left(1+\frac{1}{k-1}\right)\\
    &=\frac{1}{m}\prod_{k=2}^{m}\left(1+\frac{1}{k-1}\right),
  \end{align*}
  which implies that the inequality \eqref{6proof} holds for $n=m+1$.
  From now \eqref{5proof} and \eqref{6proof}, the desired estimate for
$|a_n|\, (n = 3,4,\ldots)$ follows, as asserted in \eqref{est an}.
This completes the proof.
\end{proof}

The problem of finding sharp upper bounds for the coefficient
functional $|a_3-\mu a_2^2|$ for different subclasses of the
normalized analytic function class $\mathcal{A}$ is known as the
Fekete-Szeg\"{o} problem.

Recently, Ali et al. \cite{ALI} considered the Fekete-Szeg\"{o}
functional associated with the $k$th root transform for several
subclasses of univalent functions. We recall here that, for a
univalent function $f(z)$ of the form \eqref{f}, the $k$th root
transform is defined by
\begin{equation}\label{F(z)}
 \mathfrak{F}(z)=[f(z^k)]^{1/k}=z+\sum_{n=1}^{\infty}b_{kn+1}z^{kn+1}\qquad
 (z\in \Delta).
\end{equation}
Following, we consider
the problem of finding sharp upper bounds for the Fekete-Szeg\"{o}
coefficient functional associated with the $k$th root transform for
functions in the class $\mathcal{BS}(\alpha)$.

In order to prove next result, we need the following lemma due to Keogh and Merkes \cite{KM}.
Further we denote by $\mathcal{P}$ the well-known class of analytic functions $p(z)$ with
$p(0) = 1$ and $\mathfrak{Re}(p(z))>0$, $z\in\Delta$.
\begin{lemma}\label{FEK}
Let the function $g(z)$ given by
\begin{equation*}
 g(z)=1+c_1z+c_2z^2+\cdots,
\end{equation*}
be in the class $\mathcal{P}$. Then, for any complex number $\mu$
\begin{equation*}
 |c_2-\mu c_1^2|\leq 2\max\{1,|2\mu-1|\}.
\end{equation*}
The result is sharp.
\end{lemma}
\begin{theorem}\label{t3.1}
Let $0\leq \alpha<1$, $f\in\mathcal{BS}(\alpha)$ and $\mathfrak{F}$ is the
$k$th root transform of $f$ defined by \eqref{F(z)}. Then, for any
complex number $\mu$,
\begin{equation}\label{1t31}
 \left|b_{2k+1}-\mu
 b_{k+1}^2\right|\leq\frac{1}{2k}\max\left\{1,\left|\frac{2(\mu-1)}{k}+1\right|\right\}.
\end{equation}
The result is sharp.
\end{theorem}
\begin{proof}
Using \eqref{Falpha}, we first put
\begin{equation}\label{FsumB}
1+F_\alpha(z)=1+\sum_{n=1}^{\infty}\mathcal{B}_n z^n,
\end{equation}
where
$\mathcal{B}_1=1$, $\mathcal{B}_2=0$, $\mathcal{B}_3=\alpha$, and etc.
Since $f\in\mathcal{BS}(\alpha)$, from Definition \ref{lemmaiff} and definition of
subordination, there exists $w\in\mathfrak{B}$ such that
\begin{equation}\label{2t31}
 zf'(z)/f(z)=1+F_{\alpha}(w(z)).
\end{equation}
We now define
\begin{equation}\label{3t31}
 p(z)=\frac{1+w(z)}{1-w(z)}=1+p_1z+p_2z^2+\cdots.
\end{equation}
Since $w\in\mathfrak{B}$, it follows that $p\in\mathcal{P}$. From \eqref{FsumB} and
\eqref{3t31} we have:
\begin{equation}\label{4t31}
 1+F_{\alpha}(w(z))=1+\frac{1}{2}\mathcal{B}_1p_1z+\left(\frac{1}{4}\mathcal{B}_2p_1^2+
 \frac{1}{2}\mathcal{B}_1\left(p_2-\frac{1}{2}p_1^2\right)\right)z^2+\cdots,
\end{equation}
where $\mathcal{B}_1=1$ and $\mathcal{B}_2=0$. Equating the coefficients of $z$ and $z^2$ on both sides of
\eqref{2t31} and substituting $\mathcal{B}_1=1$ and $\mathcal{B}_2=0$, we get
\begin{equation}\label{a2}
a_2=\frac{1}{2}p_1,
\end{equation}
and
\begin{equation}\label{a3}
 a_3=\frac{1}{8}p_1^2+\frac{1}{4}\left(p_2-\frac{1}{2}p_1^2\right).
\end{equation}
A computation shows that, for $f$ given by \eqref{f},
\begin{equation}\label{FF(z)}
 \mathfrak{F}(z)=[f(z^{1/k})]^{1/k}=z+\frac{1}{k}a_2z^{k+1}+\left(\frac{1}{k}a_3-\frac{1}{2}\frac{k-1}{k^2}a_2^2\right)z^{2k+1}+\cdots.
\end{equation}
From equations \eqref{F(z)} and \eqref{FF(z)}, we have
\begin{equation}\label{bk}
 b_{k+1}=\frac{1}{k}a_2\quad {\rm and}\quad b_{2k+1}=\frac{1}{k}a_3-\frac{1}{2}\frac{k-1}{k^2}a_2^2.
\end{equation}
Substituting from \eqref{a2} and \eqref{a3} into \eqref{bk}, we
obtain
\begin{equation*}
 b_{k+1}=\frac{1}{2k}p_1,
\end{equation*}
and
\begin{equation*}
 b_{2k+1}=\frac{1}{4k}\left(p_2-\frac{k-1}{k}p_1^2\right),
\end{equation*}
so that
\begin{equation}\label{b2k-bk}
 b_{2k+1}-\mu
 b_{k+1}^2=\frac{1}{4k}\left[p_2-\frac{1}{2}\left(\frac{2(\mu-1)}{k}+
 2\right)p_1^2\right].
\end{equation}
Letting
\begin{equation*}
 \mu'=\frac{1}{2}\left(\frac{2(\mu-1)}{k}+2\right),
\end{equation*}
the inequality \eqref{1t31} now follows as an application of Lemma
\ref{FEK}. It is easy to check that the result is sharp for the
$k$th root transforms of the function
\begin{equation}\label{fsharp}
  f(z)=z\exp\left(\int_{0}^{z}\frac{F_{\alpha}(w(t))}{t}dt\right).
\end{equation}
\end{proof}

Putting $k=1$ in Theorem \ref{t3.1}, we have:
\begin{corollary}\label{c3.1}
(Fekete-Szeg\"{o} inequality) Suppose that $f\in\mathcal{BS}(\alpha)$ and $0\leq \alpha<1$. Then, for any
complex number $\mu$,
\begin{equation}
 \left|a_3-\mu
 a_2^2\right|\leq\frac{1}{2}\max\left\{1,\left|2\mu-1\right|\right\}.
\end{equation}
The result is sharp.
\end{corollary}

It is well known that every function $f\in\mathcal{S}$ has an inverse $f^{-1}$, defined by $f^{-1}(f(z))= z$, $z\in\Delta$ and
\begin{equation*}
  f(f^{-1}(w))=w\qquad (|w|<r_0;\ \ r_0<1/4),
\end{equation*}
where
\begin{equation}\label{f-1}
  f^{-1}(w)=w-a_2w^2+(2a_2^2-a_3)w^3-(5a_2^3-5a_2a_3+a_4)w^4+\cdots.
\end{equation}
\begin{corollary}\label{c3.4}
Let the function $f$, given by \eqref{f}, be in the class $\mathcal{BS}(\alpha)$ where $0\leq \alpha<1$. Also
let the function $f^{-1}(w)=w+\sum_{n=2}^{\infty}b_nw^n$ be inverse of $f$. Then
\begin{equation}\label{b2}
  |b_2|\leq 1,
\end{equation}
and
\begin{equation}\label{b3}
  |b_3|\leq \frac{3}{2}.
\end{equation}
\end{corollary}
\begin{proof}
  Relation \eqref{f-1} give us
  \begin{equation*}
    b_2=-a_2\quad {\rm and}\quad b_3=2a_2^2-a_3.
  \end{equation*}
  Thus, we can get the estimate for $|b_2|$ by
  \begin{equation*}
    |b_2|=|a_2|\leq 1.
  \end{equation*}
  For estimate of $|b_3|$, it suffices in Corollary \ref{c3.1}, we put $\mu=2$. Hence
the proof of Corollary \ref{c3.4} is completed.
\end{proof}


\begin{thebibliography}{9}
\bibitem{ALI} Ali, R.M. Lee, S.K. Ravichandran V. and Supramanian, S.
\textit{The Fekete-Szeg\"{o} coefficient functional for transforms
of analytic functions}, Bull. Iranian Math. Soc. \textbf{35} (2009),
119--142.
\bibitem{Booth} Booth, J.
\textit{A Treatise on Some New Geometrical Methods}, Longmans, Green Reader
and Dyer, London, Vol. I (1873) and Vol. II (1877).
\bibitem{KES(Siberian)} Kargar, R. Ebadian, A. and  Sok\'{o}{\l}, J.
\textit{ On subordination of some analytic functions}, Siberian
Mathematical Journal, {\bf 57} (2016), 599--605.
\bibitem{KES(Complex)} Kargar, R. Ebadian, A. and  Sok\'{o}{\l}, J.
\textit{ Radius problems for some subclasses of analytic functions},
Complex Anal. Oper. Theory, DOI 10.1007/s11785-016-0584-x.
\bibitem{KM} Keogh, F.R. and Merkes, E.P.
\textit{ A coefficient inequality for certain classes of analytic
functions}, Proc. Amer. Math. Soc. \textbf{20} (1969), 8--12.
\bibitem{KuOwa} Kuroki, K. and S. Owa, S.
\textit{Notes on New Class for Certain Analytic Functions}, RIMS
Kokyuroku 1772, 2011, pp. 21--25.
\bibitem{psok} Piejko, K. and Sok\'{o}{\l}, J.
\textit{Hadamard product of analytic functions and
some special regions and curves}, J. Ineq. Appl. 2013, 2013:420.
\bibitem{ROB} Robertson, M.S.
\textit{ Certain classes of starlike functions}, Michigan
Mathematical Journal \textbf{76} (1954), 755--758.
\bibitem{Rog} W. Rogosinski, W.
\textit{ On the coefficients of subordinate functions}, Proc. London
Math. Soc. \textbf{48} (1943), 48--82.
\bibitem{AAS} Savelov, A.A.
\textit{ Planar curves}, Moscow (1960) (In Russian).
\end{thebibliography}
\end{document}